\documentclass[a4paper,oneside,draft]{amsart}
\title[Computing low-degree isogenies in genus 2 with the Dolgachev--Lehavi method]{ 
    Computing low-degree isogenies in genus 2\\
    with the Dolgachev--Lehavi method
}
\author{Benjamin Smith}
\address{
    INRIA Saclay--\^Ile-de-France
    and
    Laboratoire d'Informatique de l'\'Ecole polytechnique (LIX)\\
    91128 Palaiseau Cedex\\
    France
}
\email{smith@lix.polytechnique.fr}
\urladdr{http://www.lix.polytechnique.fr/~smith}

\usepackage{amsmath,amsfonts,amssymb,amsthm}
\usepackage{url}
\usepackage{fourier}

\theoremstyle{definition}
\newtheorem{definition}{Definition}[section]
\newtheorem{algorithm}[definition]{Algorithm}
\theoremstyle{remark}
\newtheorem{example}[definition]{Example}
\newtheorem{remark}[definition]{Remark}
\theoremstyle{plain}
\newtheorem{theorem}[definition]{Theorem}
\newtheorem{proposition}[definition]{Proposition}
\newtheorem{lemma}[definition]{Lemma}

\newcommand{\field}{\ensuremath{\mathbb{k}}}
\newcommand{\fieldbar}{\ensuremath{\overline{\field}}}
\newcommand{\CC}{{\mathbb{C}}}
\newcommand{\FF}{{\mathbb{F}}}
\newcommand{\ZZ}{{\mathbb{Z}}}
\newcommand{\PP}[2][{}]{\ensuremath{\mathbb{P}^{#2}_{#1}}}
\newcommand{\WPS}[2][{}]{\ensuremath{\mathbb{P}_{#1}({#2})}}
\newcommand{\Jac}[1]{\ensuremath{\mathcal{J}_{#1}}}
\newcommand{\Kum}[1]{\ensuremath{\mathcal{K}_{#1}}}

\newcommand{\subgroup}[1]{\ensuremath{\left\langle{#1}\right\rangle}}
\newcommand{\variety}[1]{\ensuremath{V\!\left({#1}\right)}}
\newcommand{\classof}[1]{\ensuremath{\left[{#1}\right]}}
\newcommand{\invol}[1]{\ensuremath{\iota_{#1}}}
\newcommand{\Mumford}[1]{\ensuremath{\langle{#1}\rangle}}
\newcommand{\spanof}[1]{\ensuremath{\left\langle{#1}\right\rangle}}

\newcommand{\Map}{\ensuremath{\Phi}}
\newcommand{\Rat}[1]{\ensuremath{\rho}_{#1}}
\newcommand{\RNC}[1]{\ensuremath{\mathcal{R}}_{#1}}
\newcommand{\HC}{\ensuremath{\mathcal{H}}}
\newcommand{\XC}{\ensuremath{\mathcal{X}}}
\newcommand{\Hplane}{\ensuremath{H}}
\newcommand{\Wspace}{\ensuremath{W}}
\newcommand{\Centre}{\ensuremath{N}}
\newcommand{\Conic}{\mathcal{Q}}
\newcommand{\Cubic}{\mathcal{C}}
\newcommand{\SConic}{\widetilde{\mathcal{Q}}}
\newcommand{\SCubic}{\widetilde{\mathcal{C}}}
\newcommand{\Secant}[1]{\ensuremath{\mathcal{L}_{#1}}}

\begin{document}

\begin{abstract}
    Let \(\ell\) be a prime,
    and \(\HC\) a curve of genus 2
    over a field \(\field\)
    of characteristic not \(2\) or \(\ell\).
    If \(S\) is a maximal Weil-isotropic subgroup of \(\Jac{\HC}[\ell]\),
    then \(\Jac{\HC}/S\) is isomorphic to the Jacobian \(\Jac{\XC}\)
    of some (possibly reducible) curve \(\XC\).
    We investigate the Dolgachev--Lehavi method
    for constructing the curve \(\XC\),
    simplifying their approach
    and making it more explicit.
    The result,
    at least for \(\ell=3\),
    is an efficient and easily programmable algorithm
    suitable for number-theoretic calculations.
\end{abstract}

\maketitle

\section{
	Introduction
}

Let $\ell \ge 3$ be prime,
and let $\HC$ be a curve of genus $2$ over
a perfect field \(\field\) of characteristic not~$2$ or~$\ell$.
Let $\Jac{\HC}$ be the Jacobian of $\HC$,
and let $S$ be a maximal $\ell$-Weil isotropic subgroup of $\Jac{H}[\ell]$;
since $\ell$ is prime,
$S \cong (\ZZ/\ell\ZZ)^2$.
The quotient $\Jac{\HC}/S$ is isomorphic 
(as a principally polarized abelian variety) 
to a Jacobian $\Jac{\XC}$,
where $\XC$ is some curve of genus~$2$
(see~\cite{Weil});
hence, there exists an isogeny 
\[
	\phi: \Jac{\HC} \to \Jac{\XC}
\]
with kernel $S$
(note that $\XC$ may be reducible,
in which case \(\Jac{\XC}\)
is a product of elliptic curves).
Our aim is to compute an explicit form for $\XC$
given $\HC$ and $S$.

In the case $\ell = 2$, the problem is resolved
by the well-known Richelot construction
(see~\cite{Bost--Mestre} 
and~\cite[Chapter 9]{Cassels--Flynn}). 
More generally, if \(\field\) is finite, 
then we can apply the explicit theta function-based algorithms of Lubicz and
Robert~\cite{Lubicz--Robert},
implemented in the freely-available \texttt{avIsogenies} package~\cite{avIsogenies}.

Alternatively, 
there is the algebraic-geometric approach described by 
Dolgachev and Lehavi~\cite{Dolgachev--Lehavi},
which computes the image of the theta divisor
on \(\Jac{\HC}\)
in the Kummer surface of \(\Jac{\XC}\).
As presented in~\cite{Dolgachev--Lehavi},
this approach has two drawbacks:
\begin{enumerate}
\item   it is not effective for \(\ell \not= 3\), and
\item   for \(\ell = 3\), where theta structures are involved,
    it assumes \(\field \subset \CC\).
\end{enumerate}    

In this work
we render the kernel of the Dolgachev--Lehavi method 
completely explicit,
with a view to
computations in number theory.
Our intention is to provide a sort of ``user's guide''
to the algorithm and its concrete implementation.
For \(\ell = 3\)
we obtain a simple, efficient, 
and easily-programmable algorithm
(that does not require \(\field \subset \CC\)).
Our algorithm retains the pleasing geometric flavour of the original,
but is better-suited to everyday calculations.

\section{
    An overview of the Dolgachev--Lehavi construction
}

We begin by briefly recalling the Dolgachev--Lehavi construction,
before treating it in detail in the following sections.
Suppose \(\HC/\field\), \(S\), \(\phi\), and \(\XC\)
are as above;
we assume we are given an explicit form for \(\HC\) and \(S\),
and we want to compute an explicit form for~\(\XC\).
Dolgachev and Lehavi observe
that if \(\Theta_{\HC}\) and \(\Theta_{\XC}\)
are theta divisors on \(\Jac{\HC}\) and \(\Jac{\XC}\),
respectively,
then \(\phi(\Theta_{\HC})\) 
is in \(|\ell\Theta_{\XC}|\)
(see~\cite[Proposition 2.4]{Dolgachev--Lehavi});
and as such, the image of \(\phi(\Theta_{\HC})\)
in the Kummer surface \(\Kum{\XC} = \Jac{\XC}/\subgroup{\pm1}\)
is a degree-\(2\ell\) rational curve\footnote{
    By ``rational curve'' we mean a curve of genus 0.
    In all other contexts, ``rational'' means ``defined over
    \(\field\)''.
}
in \(\PP{3}\)
of arithmetic genus \((\ell^2-1)/2\)
and with \((\ell^2-1)/2\) ordinary double 
points corresponding
to the nonzero elements of~\(S\), 
up to sign~\cite[Proposition 3.1]{Dolgachev--Lehavi}. 
We can compute this curve
\emph{without} knowing \(\phi\)
by expressing the map 
\(
    \Phi: \HC \cong \Theta_{\HC} \subset \Jac{\HC} 
    \to \Jac{\XC} \to \Kum{\XC} \subset \PP{3}
\)
as the composition of a double cover~\(\Rat{2\ell}\)
of a rational normal curve in \(\PP{2\ell}\)
with a projection \(\pi: \PP{2\ell} \to \PP{3}\)
whose centre depends on certain secants
corresponding
to the nonzero elements of \(S\), up to sign.
The images under \(\Phi\) of the Weierstrass points of \(\HC\)
lie on a conic \(\Conic\) contained in a hyperplane of \(\PP{3}\);
that is, a trope of \(\Kum{\XC}\).
The double cover of \(\Conic\) ramified
over the Weierstrass point images is then 
(a quadratic twist of) \(\XC\).

\section{
    The domain curve
}
\label{sec:domain}

We suppose that \(\HC/\field\)
is presented as a nonsingular projective model
\begin{equation}
\label{eq:HC}
    \HC: Y^2 = F(X,Z) = \sum_{i=0}^6 F_iX^iZ^{6-i} \subset \WPS{1,3,1},
\end{equation}
where $F$ is a squarefree homogeneous sextic over~$\field$
(such a model always exists when~$\field$ is perfect
and has characteristic not $2$: 
see~\cite[\S1.3]{Cassels--Flynn}).
The hyperelliptic involution of~\(\HC\)~is 
\[
	\invol{\HC}: (X:Y:Z)\longmapsto(X:-Y:Z) .
\]
The divisor at infinity on \(\HC\) is
\[ 
    D_{\infty} 
    = 
    \big(1:\sqrt{F_6}:0\big) + \big(1:-\sqrt{F_6}:0\big) 
    ;
\]
we observe that \(D_{\infty}\) is defined over~\field,
fixed by~\(\invol{\HC}\),
and equal to \(2(1:0:0)\) if \(F_6 = 0\).

The six Weierstrass points of $\HC$ 
are the fixed points of $\invol{\HC}$;
they
correspond to the projective roots of the sextic \(F\).
The Weierstrass divisor \(W_{\HC}\) of \(\HC\)
is the effective divisor cut out by \(Y = 0\);
if \(F(X,Z) = \prod_{i=1}^6(z_iX-x_iZ)\) over \(\fieldbar\),
then
\[
    W_{\HC} 
    = 
    (x_1:0:z_1) + \cdots + (x_6:0:z_6)
    .
\]    
Note that \(W_{\HC}\) is defined over \(\field\).
Finally,
we fix a canonical divisor 
on \(\HC\),
defining
\[
    K_{\HC} = W_{\HC} - 2D_{\infty} .
\]    

\section{
    The kernel of the isogeny
}
\label{sec:kernel}

When defining their method for \(\ell=3\),
Dolgachev and Lehavi state
``unfortunately,
we do not know how to input explicitly the pair \((\HC,S)\).
Instead we consider \(\HC\) with an odd theta structure.''
We will take a rather more middlebrow approach to the problem:
we suppose that \(\HC\) is presented in the form~\eqref{eq:HC},
and that \(S\) is given as a collection of divisor classes on
\(\HC\) expressed using an extended Mumford representation
(detailed below).

Our motivation for this choice is simple:
this is precisely how one computes with hyperelliptic Jacobians
in computational algebra systems such as Magma~\cite{Magma,Magma-JSC}
and Sage~\cite{SAGE}.
This choice also radically simplifies the algorithm:
we can omit the theta structure calculations
and pass directly to the secant computations
(short-circuiting the first four steps of
the algorithm in~\cite[\S5.1]{Dolgachev--Lehavi}).

Points on \(\Jac{\HC}\) 
correspond to divisor classes of degree zero on \(\HC\).
The Riemann--Roch theorem tells us that every nontrivial degree-$0$
class has a unique representative in the form \(P + Q - D_{\infty}\)
(this representation fails to be unique for the trivial class,
because \([P + \invol{\HC}(P) - D_{\infty}] = 0\)
for every \(P\) in \(\HC(\fieldbar)\)).

\label{sec:Mumford}
Let \(e\) be a point of \(\Jac{\HC}\),
corresponding to a divisor class \([P + Q - D_{\infty}]\).
The effective divisor \(P + Q\)
is cut out by an ideal in the form
\((A(X,Z),Y-B(X,Z))\),
where \(B\) is a homogeneous cubic
and \(A\) a homogeneous polynomial
of degree \(d \le 2\).
The triple
\[\Mumford{a(x),b(x),d} := \Mumford{A(x,1),B(x,1),d}\]
then encodes the point \(e\)
(with the convention that \(A\) is chosen such that \(a\) is monic).
Note that if \(\field'\)
is an extension of \(\field\),
then \(e = \Mumford{a,b,d}\)
is in \(\Jac{\HC}(\field')\)
if and only if \(a\) and \(b\) have coefficients in \(\field'\).

Conversely,
given a triple \(\Mumford{a,b,d}\),
we recover the corresponding point of \(\Jac{\HC}\)
by computing the effective divisor cut out by \((A(X,Z),Y-B(X,Z))\),
where \(B\) is the degree-$3$ homogenization of \(b\)
and \(A\) is the degree-$d$ homogenization of \(a\),
and then subtracting \((d/2)D_{\infty}\).
If \(\HC\) has two points at infinity
(that is, if \(F_6 \not= 0\))
then \(d\) must be either~\(2\) or~\(0\).
In the case where \(\HC\) has a single point at infinity
(that is, when \(F_6 = 0\))
we always have \(d = \deg a\),
and the pair \(\Mumford{a,b \bmod a}\) is the standard Mumford representation.
The advantage of the extended representation above
is that it gracefully handles the general case
where there are two points at infinity.
\begin{example}
    Consider the following points on the Jacobian of \(\HC: Y^2 = X^6 - Z^6\).
    \begin{itemize}
    \item   \(0\)
        is represented by \(\Mumford{1,0,0}\).
    \item   
        \([(1:0:1) + (-1:0:1) - D_{\infty}]\)
        is represented by \(\Mumford{x^2-1,0,2}\).
    \item   
        \([(1:1:0) - (1:-1:0)] = [2(1:1:0)-D_{\infty}] \)
        is represented by \(\Mumford{1,x^3,2}\).
    \end{itemize}
\end{example}

In this article,
we will assume that the points of \(S\)
are all \(\field\)-rational.
This simplifies the exposition and the computations;
however, all of our calculations are symmetric in the elements of \(S\).
The algorithm should therefore be easily adapted to the case
where \(S\) is rational but its elements are not.

\section{
    The rational normal curve
}
\label{sec:rat}

The Riemann--Roch space \(L(2\ell K_{\HC})\)
is a direct sum of subspaces
\[
    L(2\ell K_{\HC})
    =
    L(2\ell K_{\HC})^+
    \oplus
    L(2\ell K_{\HC})^-
    ,
\]   
where 
$\invol{\HC}$ acts as $+1$ on the elements of $L(2\ell K_{\HC})^+$
and $-1$ on the elements of $L(2\ell K_{\HC})^-$.
Writing 
\(x = X/Z\) and \(y = Y/Z^3\),
we have
\[
        L(2\ell K_{\HC})^+ 
        =
        \big\langle
            x^i/y^{2\ell}
        \big\rangle_{i=0}^{2\ell}
        \quad \text{ and } \quad 
        L(2\ell K_{\HC})^- 
        =
        \big\langle
            x^i/y^{2\ell-1} 
        \big\rangle_{i=0}^{2\ell-3}
        .
\]    

The space \(L(2\ell K_{\HC})^+\)
corresponds to the linear system
\(|2\ell K_{\HC}|^{\langle\invol{\HC}\rangle}\);
we see immediately that it is \((2\ell+1)\)-dimensional,
and therefore defines a map 
\[
    \Rat{2\ell}: \HC \longrightarrow \RNC{2\ell} \subset \PP{2\ell}
\]
onto a curve \(\RNC{2\ell}\) in \(\PP{2\ell}\).
Fixing coordinates on \(\PP{2\ell}\),
we take \(\Rat{2\ell}\) to be defined by
\[
    \Rat{2\ell}: (X:Y:Z) \longmapsto
    (U_0:\cdots:U_{2\ell})
    =
    (X^0Z^{2\ell}:XZ^{2\ell-1}:\cdots:X^{2\ell-1}Z:X^{2\ell})
    .
\]
We see that
$\RNC{2\ell}$ is a rational normal curve of degree \(2\ell\) in $\PP{2\ell}$,
and
$\Rat{2\ell}$ is a double cover:
\begin{equation}
\label{eq:RNE-coincidence}
    \Rat{2\ell}(P) = \Rat{2\ell}(Q) \iff 
    \big(P = Q\ \ \text{ or }\ \ P = \invol{\HC}(Q)\big)
    .
\end{equation}    
(Essentially,
\(\Rat{2\ell}\)
is a composition of the canonical map of \(\HC\)
and an \(\ell\)-uple embedding.)

\section{
    The secant lines
}
\label{sec:secants}

We adopt the following convention:
if $S$ is a set of points in some projective space $\PP{n}$,
then $\spanof{S}$ denotes the linear subspace of $\PP{n}$
generated by $S$.

For any pair of points $P$ and $Q$ on $\HC$,
we define $\Secant{P,Q}$
to be the line in \(\PP{2\ell}\)
intersecting \(\RNC{2\ell}\)
in \(\Rat{2\ell}(P) + \Rat{2\ell}(Q)\);
that is, 
\[
    \Secant{P,Q} := \left\{
    \begin{array}{ll}
        \spanof{\Rat{2\ell}(P),\Rat{2\ell}(Q)}
            & \text{ if } P \notin \{Q,\invol{\HC}(Q)\} 
        \\
        T_{\Rat{2\ell}(P)}(\RNC{2\ell}) 
            & \text{ otherwise. }
    \end{array}        
    \right.
\]        
We can also define secant lines corresponding to nonzero Jacobian elements:
if \(e\) is a nonzero point on $\Jac{\HC}$,
then we define 
\[
    \Secant{e} 
    := \Secant{P,Q} 
    \quad \text{ where } 
    e = \classof{P + Q - D_{\infty}}
    .
\]    
Observe that
\(
    \Secant{P,Q} 
    = 
    \Secant{P,\invol{\HC}(Q)} 
    = 
    \Secant{\invol{\HC}(P),Q} 
    = 
    \Secant{\invol{\HC}(P),\invol{\HC}(Q)}
\)    
for all \(P\) and \(Q\) on \(\HC\),
so 
\[
    \Secant{e} = \Secant{-e}
\]
for all \(e\) in~\(\Jac{\HC}\setminus\{0\}\).

\begin{remark}
    Dolgachev and Lehavi
    define 
    secants \(l_e = \spanof{\Rat{2\ell}(P_1),\Rat{2\ell}(P_2)}\)
    for each nontrivial point \(e = [P_1 - P_2]\)
    in \(\Jac{\HC}\)
    (see~\cite[Theorem~1.1]{Dolgachev--Lehavi}).
    Our \(\Secant{e}\)
    is equal to \(l_e\),
    because \([P_1 - P_2] = [P_1 + \invol{\HC}(P_2) - D_\infty]\)
    and \(\Secant{P_1,P_2} = \Secant{P_1,\invol{\HC}(P_2)}\).
\end{remark}    

The following lemma gives explicit and rational formul\ae{}
for the secants \(\Secant{e}\) and their intersections
with arbitrary hyperplanes in \(\PP{2}\).
These formul\ae{} are central to the explicit Dolgachev--Lehavi method.

\begin{lemma}
\label{lemma:H-Le-intersection}
    Let \(e = \Mumford{a,b,d}\) be a nonzero point of \(\Jac{\HC}\).
    Let \(H: \sum_{i=0}^{2\ell} H_i U_i = 0\) be a hyperplane in
    \(\PP{2\ell}\),
    and write \(h(x) := \sum_{i=0}^{2\ell} H_ix^i \).
    \begin{enumerate}    
    \item   If \(a = 1\) and \(d = 2\), then
        \[ 
            \Secant{e} = \big\langle{
                (0:\cdots:0:1:0),(0:\cdots:0:0:1)
            }\big\rangle . 
        \]
        \begin{enumerate}
        \item   If \(H_{2\ell} = H_{2\ell-1} = 0\), 
            then \(\Secant{e} \subset H\).
        \item   Otherwise,
            \( \Secant{e}\cap H = (0:\cdots:0:H_{2\ell}:-H_{2\ell-1}) \).
        \end{enumerate}
    \item   If \(a(x) = x - \alpha\), then
        \[ 
            \Secant{e} = \big\langle{
                (0:\cdots:0:1),(1:\cdots:\alpha^{2\ell})
            }\big\rangle
            .
        \]
        \begin{enumerate}
        \item   If \(h(\alpha) = 0\)
            and \(H_{2\ell} = 0\), then
            \(\Secant{e} \subset H\).
        \item   Otherwise,
            \( 
                \Secant{e}\cap H 
                =
                \left(H_{2\ell}:H_{2\ell}\alpha:\cdots:H_{2\ell}\alpha^{2\ell-1}:
                    H_{2\ell}\alpha^{2\ell} - h(\alpha)\right)
            \).       
        \end{enumerate}
    \item   If \(a(x) = x^2 +a_1x+a_0\) with \(a_1^2\not=4a_0\), then
        \[
            \Secant{e} = 
            \big\langle{
                (\pi_0:\cdots:\pi_{2\ell}),
                (-a_1:a_2\pi_0:a_2\pi_1:\cdots:a_2\pi_{2\ell-1})
            }\big\rangle
        \]
        where
        \(\pi_0 = 2\), \(\pi_1 = -a_1\),
        and \(\pi_{i} = -a_1\pi_{i-1} - a_2\pi_{i-2}\)
        for \(i > 1\).
        \begin{enumerate}
        \item   If \(a(x)\) divides \(h(x)\), then \(\Secant{e} \subset H\).
        \item   Otherwise,
            \(
                \Secant{e}\cap H
                =
                (\gamma_0:\cdots:\gamma_{2\ell})
            \)
            where
            \[ 
                \gamma_i = 
                    \sum_{0 \le j \le 2\ell}
                        H_j(a_2^j\sigma_{i-j} - a_2^i\sigma_{j-i})
            \] 
            with
            \(\sigma_k = 0\) for \(k < 1\),
            \( \sigma_1 = 1 \), 
            and
            \( \sigma_k = -a_1\sigma_{k-1} - a_2\sigma_{k-2} \)
            for \(k > 1\).
        \end{enumerate}
    \item   If \(a(x) = x^2 + a_1x + a_0\) with \(a_1^2 = 4a_0\), then
        writing \(\alpha\) for \(-a_1/2\), we have
        \[
            \Secant{e}
            =
            \big\langle{
                (1:\alpha:\cdots:\alpha^{2\ell}),
                (0:1:2\alpha:\cdots:2\ell\alpha^{2\ell-1})
            }\big\rangle
            .
        \]    
        \begin{enumerate}
        \item   If \(a(x)\) divides \(h(x)\),
            then \(\Secant{e} \subset H\).
        \item   Otherwise,    
            \(
                \Secant{e}\cap H
                =
                (\gamma_0:\cdots:\gamma_{2\ell})
            \)
            where
            \( \gamma_i = i\alpha^{i-1}h(\alpha)-\alpha^ih'(\alpha) \).
        \end{enumerate}
    \end{enumerate}    
\end{lemma}    
\begin{proof}
    In general,
    given points \(\alpha = (\alpha_0:\cdots:\alpha_{2\ell})\)
    and \(\beta = (\beta_0:\cdots:\beta_{2\ell})\)
    in \(\PP{2\ell}\),
    we have 
    \[
        H\cap\Secant{\alpha,\beta} 
        = 
        (A\beta_0 - B\alpha_0:\cdots:A\beta_{2\ell}-B\alpha_{2\ell})
    \]    
    where \(A = \sum_{i=0}^{2\ell}H_i\alpha_i\)
    and \(B = \sum_{i=0}^{2\ell}H_i\beta_i\);
    if \(A = B = 0\),
    then \(\Secant{\alpha,\beta} \subset H\)
    (and the point above is not defined).
    In the following,
    we suppose \(e = [P + Q - D_{\infty}]\);
    we have \(e \not= 0\),
    so we can suppose \(P \not= \invol{\HC}(Q)\).

    In case (1)
    both \(P\) and \(Q\) are at infinity,
    so \(\Rat{2\ell}(P) = \Rat{2\ell}(Q) = (0:\cdots:0:1)\);
    our expression for \(\Secant{e}\)
    gives generators for the tangent 
    to \(\RNC{2\ell}\) at \((0:\cdots:0:1)\).
    The intersection formula follows immediately.

    In case (2),
    we have
    \(P = (1:0:0)\)
    and
    \(Q = (\alpha:\pm\sqrt{F(\alpha,1)}:1)\),
    so \(\Rat{2\ell}(P) = (0:\cdots:0:1)\)
    and \(\Rat{2\ell}(Q) = (1:\alpha:\cdots:\alpha^{2\ell})\).
    The intersection formula follows immediately.

    In case (3),
    we have
    \(P = (\alpha:\pm\sqrt{F(\alpha,1)}:1)\) 
    and
    \(Q = (\beta:\pm\sqrt{F(\beta,1)}:1)\)
    with \(\alpha \not= \beta\),
    \(\alpha + \beta = -a_1\),
    and \(\alpha\beta = a_2\);
    so \(\Rat{2\ell}(P) = (1:\alpha:\cdots:\alpha^{2\ell})\)
    and \(\Rat{2\ell}(Q) = (1:\beta:\cdots:\beta^{2\ell})\).
    If we take
    \[
        T
        = 
        (2:\alpha+\beta:\cdots:\alpha^{2\ell} + \beta^{2\ell})
        \quad
        \text{and}
        \quad
        S
        = 
        (\alpha+\beta:2\beta\alpha:\cdots:\alpha\beta^{2\ell} + \beta\alpha^{2\ell})
        ,
    \]    
    then we easily verify that
    \(\Secant{P,Q} = \Secant{T,S}\);
    it is a straightforward exercise with symmetric polynomials
    to show that 
    \(\alpha^i + \beta^i = \pi_i\) for \(0 \le i \le 2\ell\)
    and \(\alpha\beta^i + \beta\alpha^i = a_2\pi_{i-1}\) for \(i > 0\),
    whence our formula for \(\Secant{e}\).
    The intersection \(H\cap\Secant{e}\) is 
    \[
        H\cap\Secant{P,Q}
        =
        (
         h(\alpha) - h(\beta)
        :h(\alpha)\beta-h(\beta)\alpha
        :\cdots
        :h(\alpha)\beta^{2\ell}-h(\beta)\alpha^{2\ell}
        )
        ;
    \]   
    it is another straightforward exercise to show that
    \[
        \alpha^j\beta^i-\beta^j\alpha^i 
        =
        (\beta - \alpha)(a_2^j\sigma_{i-j}-a_2^i\sigma_{j-i})
        ,
    \]
    so
    \(
        h(\alpha)\beta^i - h(\beta)\alpha^i 
        = 
        \sum_{j=0}^{2\ell}H_j(\beta-\alpha)(a_2^j\sigma_{i-j}-a_2^i\sigma_{j-i})
        = 
        (\beta-\alpha)\gamma_i
    \)    
    for \(0 \le i \le 2\ell\),
    and thus \(H\cap\Secant{e} = (\gamma_0:\cdots:\gamma_{2\ell})\).

    In case (4),
    we have
    \(P = Q = (\alpha:\pm\sqrt{F(\alpha,1)}:1)\);
    our expression for \(\Secant{e}\)
    gives generators for the tangent to \(\RNC{2\ell}\)
    at \(\Rat{2\ell}(P) = (1:\alpha:\cdots:\alpha^{2\ell})\).
    The intersection formula follows.
\end{proof}

\section{
    The Weierstrass subspace 
}

Since \(\RNC{2\ell}\)
is a rational normal curve of degree \(2\ell\),
any $2\ell+1$ distinct points on $\RNC{2\ell}$ are linearly independent.
In particular,
the images of the six Weierstrass points of $\HC$ under $\Rat{2\ell}$
are linearly independent because $\ell \ge 3$.
In view of~\eqref{eq:RNE-coincidence} 
the images are distinct, so the subspace
\[ 
    \Wspace := \spanof{\Rat{2\ell}(W_{\HC})} \subset \PP{2\ell} 
\]
is five-dimensional.

For each
\( 0 \le i \le 2\ell-6 \),
we define a linear form
\[
    W_i := \sum_{j = 0}^6 F_jU_{i+j} 
    .
\]    

\begin{lemma}
\label{lemma:Wspace}
    The space $\Wspace$ 
    is 
    \[
        \Wspace
        =
        \bigcap_{i=0}^{2\ell-6}\variety{W_i}
        =
        \variety{\left\{ W_i : 0 \le i \le 2\ell-6 \right\}}
        .
    \]
\end{lemma}
\begin{proof}
    Each hyperplane $\variety{W_i}$ 
    contains $\Wspace$,
    since $W_i\circ\Rat{2\ell} = X^iZ^{2\ell-6-i}F(X,Z)$.
    But the~$W_i$ are linearly independent,
    so the intersection $\bigcap_{i=0}^{2\ell-6}\variety{W_i}$
    is $5$-dimensional, and hence equal to $\Wspace$.
\end{proof}

\section{
    The theorem of Dolgachev and Lehavi
}
\label{sec:DL}

We are now ready to state the main theorem
behind the Dolgachev--Lehavi method.

\begin{theorem}[\protect{\cite[Theorem 1.1]{Dolgachev--Lehavi}}]
\label{theorem:DL}
	There exists a unique hyperplane $\Hplane \subset \PP{2\ell}$ such that
    \begin{enumerate}
    \item   $\Hplane$ contains $\Wspace$, and
	\item   the intersection points of $\Hplane$ with the secants
            $\Secant{e}$
            for each nonzero $e$ in $S$
	        are contained in a subspace $\Centre$ 
            of codimension $3$ in \(\Hplane\).
    \end{enumerate}       
	The image of the Weierstrass divisor under the projection 
    \( \PP{2\ell} \to \PP{3} \) with centre~$\Centre$
	lies on a conic \(\Conic\) (which may be reducible),
	and the double cover of \(\Conic\) ramified over this divisor
	is a stable curve $\XC$ of arithmetic genus $2$
	such that $\Jac{\XC} \cong \Jac{\HC}/S$.
\end{theorem}

It is crucial to note that Theorem~\ref{theorem:DL}
is not constructive:
it does not in itself yield the hyperplane \(\Hplane\),
nor the centre \(\Centre\) of the projection to \(\PP{3}\).
It is noted 
in~\cite[\S3.4]{Dolgachev--Lehavi}
that \(\Hplane\) is defined by \(\phi^*(\Theta_{\XC})\),
but in our situation we do not yet have an expression for \(\phi\)
or~\(\Theta_{\XC}\).

In the case \(\ell = 3\),
we are saved by a happy coincidence: \(2\ell-1 = 5\),
so \(\Hplane = \Wspace\)
(we return to this case in~\S\ref{sec:ell=3} below).
For \(\ell > 3\),
we must compute \(\Hplane\) in some other way;
Lemma~\ref{lemma:Hplanes},
an easy corollary of Lemma~\ref{lemma:Wspace},
characterizes the possible hyperplanes.

\begin{lemma}
\label{lemma:Hplanes}
    The linear system of all hyperplanes in \(\PP{2\ell}\)
    containing \(\Wspace\)
    is generated by the \(2\ell-5\) hyperplanes \(\variety{W_i}\)
    for \(0 \le i \le 2\ell-6\).
    That is,
    if \(H \supset W\) is a hyperplane in~\(\PP{2\ell}\), then
    \[
        \Hplane = \variety{
            \alpha_0 W_0 + \cdots + \alpha_{2\ell-6} W_{2\ell-6}
        }
    \]
    for some \((\alpha_0:\cdots:\alpha_{2\ell-6})\) in
    \(\PP{2\ell-6}(\field)\).
\end{lemma}

In view of Lemma~\ref{lemma:Hplanes},
one na\"ive approach to computing \(\Hplane\)
for \(\ell > 3\)
would be to take a generic 
\( \Hplane = \variety{ \sum_{i=0}^{2\ell-6} \alpha_i W_i } \)
and compute its intersection
with the secants \(\Secant{e}\).
This yields \((\ell^2-1)/2\) points 
whose coordinates are linear expressions 
in the \(\alpha_i\).
We could then solve for the values of \(\alpha_i\)
by computing the zero locus of the \((2\ell-2)\times(2\ell-2)\) minors of
the matrix formed by the intersections \(\Hplane\cap\Secant{e}\);
but each minor is still a degree-\((2\ell-2)\) polynomial
in \(2\ell-5\) variables, and the number of minors is exponential in \(\ell\).

Alternatively,
we could take a generic set of linear equations
determining \(\Centre\) inside the generic \(\Hplane\);
requiring that this centre intersects any one of the \((\ell^2-1)/2\)
secants imposes \(O(\ell^4)\) quartic polynomial conditions 
on the \(O(\ell)\) unknowns.

In each approach the system is highly overdetermined, 
and with a clever choice of minors
we might hope to get lucky and find solutions for toy examples.
However,
both approaches already represent a significant undertaking
for \(\ell = 5\), even over finite fields;
they are totally impractical for larger \(\ell\)
and for infinite fields.

We continue the treatment for general \(\ell\) 
in \S\ref{sec:Phi} and \S\ref{sec:codomain},
supposing that an equation for \(\Hplane\) has been found;
without such an equation,
the \texttt{avIsogenies} package~\cite{avIsogenies}
represents a much more sensible approach for \(\ell \ge 5\)
(if \(\field\) is finite).
For \(\ell = 3\), the Dolgachev--Lehavi method is 
as practical as it is interesting;
we specialize to this case in \S\ref{sec:ell=3} and~\S\ref{sec:example}.

\section{
    From theory to practice
}
\label{sec:Phi}

To compute \(\XC\)
via Theorem~\ref{theorem:DL},
we must compute the map
\[
    \Phi := \pi\circ\Rat{2\ell} : \HC \to \PP{3} ,
\]
where \(\pi: \PP{2\ell} \to \PP{3}\) is the projection with centre \(\Centre\).
Suppose that we have 
an equation 
\[
    \Hplane: \sum_i\alpha_iW_i = 0
\] 
for \(\Hplane\).
We can then apply Lemma~\ref{lemma:H-Le-intersection}
to compute the centre 
\(
    \Centre 
    = 
    \spanof{\Secant{e}\cap \Hplane : e \in S\setminus\{0\}}
\).
Since \(\Centre\subset\Hplane\),
we may compute 
\(
    \nu_{0,0},\ldots,\nu_{0,2\ell},
    \nu_{1,0},\ldots,\nu_{1,2\ell},
    \nu_{2,0},\ldots,\nu_{2,2\ell}
\)
in \(\field\)
such that 
\[
    \Centre = \variety{
        \sum_{i=0}^{2\ell} \nu_{0,i}U_i , \ 
        \sum_{i=0}^{2\ell} \nu_{1,i}U_i , \ 
        \sum_{i=0}^{2\ell} \nu_{2,i}U_i , \ 
        \sum_{i=0}^{2\ell-6} \alpha_iW_i  
    }.
\]
(This amounts to computing the kernel of the matrix
whose rows are formed by the coordinates of the
\(\Secant{e}\cap\Hplane\);
the choice of \(\sum_{i=0}^6\alpha_iW_i\)
for the fourth defining equation
will be convenient later in the procedure.)

Fixing coordinates on \(\PP{3}\),
the projection \(\pi\)
with centre \(\Centre\)
is defined by
\[
    \pi: (U_0:\cdots:U_{2\ell})
    \longmapsto
    (V_0:V_1:V_2:V_3)
    =
    \left(
        \sum_{i=0}^{2\ell} \nu_{0,i}U_i , \ 
        \sum_{i=0}^{2\ell} \nu_{1,i}U_i , \ 
        \sum_{i=0}^{2\ell} \nu_{2,i}U_i , \ 
        \sum_{i=0}^{2\ell-6} \alpha_iW_i  
    \right)
    ;
\]    
the composed map \(\Map = \pi\circ\Rat{2\ell}\)
is then
\[
    \Map :
    (X:Y:Z)
    \longmapsto
    (V_0:V_1:V_2:V_3)
    =
    \big(
        \Map_0(X,Z) : \Map_1(X,Z) : \Map_2(X,Z) : \Map_3(X,Z)
    \big)
    ,
\]    
where
\[
    \Map_0 := \sum_{i=0}^{2\ell} \nu_{0,i}X^iZ^{2\ell-i},
    \quad 
    \Map_2 := \sum_{i=0}^{2\ell} \nu_{1,i}X^iZ^{2\ell-i},
    \quad 
    \Map_2 := \sum_{i=0}^{2\ell} \nu_{2,i}X^iZ^{2\ell-i},
\]
and
\[
    \Map_3 := \sum_{i=0}^{2\ell-6} \alpha_iX^iZ^{2\ell-6-i}F(X,Z)
    .
\]

The image of \(\Map\)
is a rational curve of degree \(2\ell\) in \(\PP{3}\).
It lies on the Kummer surface \(\Kum{\XC}\) of the unknown
codomain Jacobian \(\Jac{\XC}\),
and is therefore the intersection of a quadric and a cubic hypersurface
in \(\PP{3}\) (see~\cite[Chapter XIII]{Hudson}):
\[
    \Map(\PP{1})
    =
    \SConic \cap \SCubic 
    \quad
    \text{where}
    \quad 
    \SConic = \variety{\tilde Q(V_0,V_1,V_2,V_3)}
    \quad
    \text{and}
    \quad 
    \SCubic = \variety{\tilde C(V_0,V_1,V_2,V_3)}
\]
for some forms $\tilde Q$ and $\tilde C$ of degree $2$ and $3$,
respectively.
The forms
\(\tilde Q\) and \(\tilde C\)
generate the elimination ideal
\[
    \big(\tilde Q,\tilde C\big) = 
    (V_0-\Map_0,V_1-\Map_1,V_2-\Map_2,V_3-\Map_3)
    \cap\field[V_0,V_1,V_2,V_3]
    ;
\]
note that \(\tilde Q\) is uniquely determined,
and \(\tilde C\) is determined modulo \((V_0Q,V_1Q,V_2Q,V_3Q)\).

The Weierstrass points of \(\HC\)
map into the hyperplane \(V_3 = 0\),
which we identify with~\(\PP{2}\).
(This simplification motivates our choice of \(\Phi_3\).)
Theorem~\ref{theorem:DL}
asserts that a conic \(\Conic\) passes through the six images,
and indeed
\[
    \Conic = \variety{Q(V_0,V_1,V_2)} \subset \PP{2} ,
    \quad 
    \text{where}
    \quad 
    Q(V_0,V_1,V_2) = \tilde Q(V_0,V_1,V_2,0) .
\]    
The image of the Weierstrass divisor under \(\Phi\) is therefore
\( \Conic\cap\Cubic \),
where 
\[    
    \Cubic = \variety{C(V_0,V_1,V_2)} \subset \PP{2}
    \quad 
    \text{with}
    \quad 
    C(V_0,V_1,V_2) = \tilde C(V_0,V_1,V_2,0)
    .
\]    
We are more interested in the forms \(Q\) and \(C\)
than in \(\tilde Q\) and \(\tilde C\),
and it is a simple matter to interpolate them.
For \(Q\),
we compute the six quintic polynomials
\(\Phi_i\Phi_j(x,1)\bmod F(x,1)\) for \(0\le i\le j\le 2\);
the unique linear relation between them
(and between the \(\nu_{i,0}\nu_{j,0}\)
if \(F_6 = 0\))
yields the coefficients of \(Q\).
Similarly,
to find \(C\)
we compute the ten quintics
\(\Phi_i\Phi_j\Phi_k(x,1)\bmod F(x,1)\) for \(0\le i\le j\le k\le 2\);
any one of the linear relations between them 
(and the \(\nu_{i,0}\nu_{j,0}\nu_{k,0}\) if \(F_6 = 0\))
gives an equation
for a valid cubic \(C\).

\section{
    The codomain curve
}
\label{sec:codomain}

The data \((\Conic,\Conic\cap\Cubic)\)
specifies a genus 2 curve \(\XC\)
(up to a quadratic twist)
as a double cover
of $\Conic$ ramified over the six
points of $\Conic \cap \Cubic$.
This is the output of the Dolgachev--Lehavi algorithm
and of Theorem~\ref{theorem:DL},
and it is sufficient for computing isomorphism invariants of \(\XC\)
(see, for example, \cite{Cardona--Quer} and \cite{Mestre}).

In some situations, however,
we would like to derive a defining equation for \(\XC\) itself.
When~$\Conic$ is nonsingular,
we recover a hyperelliptic curve;
in the degenerate case where \(\Conic\) is singular,
we recover a union of two elliptic curves
$\XC_+$ and $\XC_-$,
which are generally defined over
a quadratic extension of $\field$
(in which case they are Galois conjugates).
The procedure is essentially standard (cf.~\cite[\S2]{Cardona--Quer}),
but we recall it here for completeness.

\begin{algorithm}
\label{algorithm:curve-recovery}
Computes a (possibly reducible) genus 2 curve
representing a double cover of a given plane conic
ramified over the intersection with a plane cubic.
\begin{description}
\item[Input]
    A plane conic 
    \(\Conic: Q(V_0,V_1,V_2) = 0 \)
    and cubic \( \Cubic: C(V_0,V_1,V_2) = 0 \).
\item[Output]
    A genus 2 curve \(\XC\) forming a double cover of \(\Conic\)
    ramified over \(\Conic\cap\Cubic\).
    If \(\Conic\) is singular,
    then \(\XC\)
    will be a one-point union of elliptic curves \(\XC_+\)
    and \(\XC_-\),
    with \(\XC_\pm\)
    ramified over \(P_0\) 
    and \(\Cubic\cap\mathcal{L}_\pm\)
    where \(\Conic = \mathcal{L}_+ + \mathcal{L}_-\) 
    and \(P_0 = \mathcal{L}_+ \cap \mathcal{L}_-\).
\item[1]
    Let \(M\) be the matrix defined by
    \[
        \qquad
        \qquad
        M := 
        \left(
        \begin{array}{ccc}
            2q_{0,0} & q_{0,1} & q_{0,2} \\
            q_{0,1} & 2q_{1,1} & q_{1,2} \\
            q_{0,2} & q_{1,2} & 2q_{2,2} \\
        \end{array}
        \right)
        ,
        \quad 
        \text{where}
        \quad 
        \sum_{0\le i\le j\le 2}q_{i,j}V_iV_j = Q(V_0,V_1,V_2) .
    \]
\item[2]
    If \(\det(M) = 0\), then \(\Conic\) is singular.
    \begin{description}
    \item[2a]
        Compute a diagonal matrix \(D = \mathrm{diag}(a,b,0)\) 
        and an invertible matrix \(T\) 
        such that
        \( M = T D T^{-1} \).
    \item[2b]
        Set \(\delta = \sqrt{-a/b}\),
        and define
        homogeneous cubics 
        \(C_+(X,Z)\) and \(C_-(X,Z)\) by
        \(
            C_{\pm} := 
            C\big(
                (t_{00}\pm\delta t_{01})Z + t_{02}X,
                (t_{10}\pm\delta t_{11})Z + t_{12}X,
                (t_{20}\pm\delta t_{21})Z + t_{22}X
            \big)    
        \)
        where
        \[
            \left(
            \begin{array}{ccc}
                t_{00} & t_{01} & t_{02} \\
                t_{10} & t_{11} & t_{12} \\
                t_{20} & t_{21} & t_{22} \\
            \end{array}
            \right)
            =
            T
            .
        \]    
    \item[2c]
        Define elliptic curves
        \(\XC_+\)
        and
        \(\XC_-\)
        over \(\field(\delta)\)
        in \(\WPS{2,3,2}\)
        by
        \[
            \XC_+: Y^2 = C_+(X,Z)
            \quad \text{and}\quad 
            \XC_-: Y^2 = C_-(X,Z)
            ,
        \]
        and return the union of \(\XC_+\) and \(\XC_-\)
        identifying the points at infinity.
    \end{description}
\item[3]
    Otherwise, \(\Conic\) is nonsingular.
    \begin{description}
    \item[3a]
        Compute a rational point \(P = (\alpha_0:\alpha_1:\alpha_2)\) in
        \(\Conic(\field)\) (see~Remark~\ref{remark:conic}).
    \item[3b]    
        Let \(\pi: \PP{1} \to \Conic\)
        be the corresponding 
        parametrization,
        defined by
        \[
            \qquad
            \pi: (X:Z) \longmapsto (V_0:V_1:V_2) = (P_0(X,Z):P_1(X,Z):P_2(X,Z))
        \]    
        (the \(P_i\) are quadratic forms).
    \item[3c]
        Return \(\XC: Y^2 = C(P_0(X,Z),P_1(X,Z),P_2(X,Z))\).
    \end{description}
\end{description}
\end{algorithm}

\begin{remark}
\label{remark:conic}
    Step 3a of Algorithm~\ref{algorithm:curve-recovery}
    requires us to compute a \(\field\)-rational point \(P\)
    on the conic \(\Conic\).
    If \(\HC\) has a rational Weierstrass point \(W_0\),
    then we may take \(P = \Psi(W_0)\).
    Generically, however,
    \(\HC\) has no rational Weierstrass points,
    and then we are obliged to search for a rational point on
    \(\Conic\).
    We are guaranteed that such a rational point exists
    (cf.~\cite[Lemme 1]{Mestre}).
    Over a finite field,
    finding a rational point is straightforward;
    over the rationals,
    we can apply (for example)
    the Cremona--Rusin algorithm~\cite{Cremona--Rusin}.
\end{remark}

\section{
	The algorithm for \(\ell = 3\)
}
\label{sec:ell=3}

Consider now the special case $\ell = 3$.
The map
$\Rat{6} : \HC \to \RNC{6} \subset \PP{6}$
is defined by 
\[
    \Rat{6}:
    (X:Y:Z) 
    \longmapsto 
    (U_0:U_1:\cdots:U_5:U_6) 
    =
    \big(Z^6:XZ^5:\cdots:X^5Z:X^6\big) 
    .
\]
The hyperplane $\Hplane$ of Theorem~\ref{theorem:DL}
contains $\Wspace = \spanof{\Rat{6}(W_{\HC})}$ by definition;
but $\dim\Hplane=\dim\Wspace=5$,
so 
\( \Hplane = \Wspace \).
Applying Lemma~\ref{lemma:Wspace},
we find
\begin{equation}
\label{eq:ell=3-Hplane}
    \Hplane = \variety{W_0} = \variety{\sum_{i=0}^6F_iU_i} \subset \PP{6} ;
\end{equation} 
this allows us to simplify Lemma~\ref{lemma:H-Le-intersection}
for the case \(\ell = 3\).

\begin{proposition}
\label{proposition:H-Le-ell=3}
    If \(e = \Mumford{a,b,d}\) is a nonzero 3-torsion point of \(\Jac{\HC}\),
    then
    \[
        \Hplane\cap\Secant{e}
        =
        \big(\gamma_0(e):\cdots:\gamma_6(e)\big)
        ,
    \]    
    where the $\gamma_i$ are defined as follows:
    \begin{enumerate}
    \item
        If \(a = 1\),
        then \(\gamma_i(e) = 0\) for \(0 \le i < 5\),
        with \(\gamma_5(e) = F_6\) and \(\gamma_6(e) = -F_5\).
    \item   
        If \(a\) is linear,
        then 
        \( \gamma_i(e) = 0 \) for \(0 \le i < 6 \),
        and \( \gamma_6(e) = 1 \).
    \item    
        If \( a(x) = x^2 + a_1x + a_0 \)
        with \(a_1^2 \not= 4a_0\),
        then
        \[
            \gamma_i(e) 
            = 
            \sum_{j = 0}^{6}
                F_j(a_2^j\sigma_{i-j} + a_2^i\sigma_{j-i})
            \ \ \text{for}\ \ 
            0 \le i \le 6 
        \] 
        with
        \(\sigma_k = 0\) for \(k < 1\),
        \( \sigma_1 = 1 \), 
        and
        \( \sigma_k = -a_1\sigma_{k-1} - a_2\sigma_{k-2} \)
        for \(k > 1\).
    \item   If \(a(x) = x^2 + a_1x + a_0\) with \(a_1^2 = 4a_0\), then
        \[
            \gamma_i(e) =
            \sum_{j=0}^{6}(i-j)F_j(-a_1/2)^{i+j-1}
            \ \ \text{for}\ \ 
            0 \le i \le 6 
            .
        \]   
     \end{enumerate}   
\end{proposition}
\begin{proof}
    This follows immediately from Lemma~\ref{lemma:H-Le-intersection}
    on setting \(H = \variety{\sum_{i=0}^6F_iU_i}\)
    and noting that
    \(a(x)\) cannot divide \(h(x) = \sum_{i=0}^6F_ix^i\)
    (since otherwise \(e\) would have order \(2\)).
\end{proof}

We now present
a version of the Dolgachev--Lehavi algorithm for \(\ell=3\)
based on the extended Mumford representation.
The algorithm requires only elementary matrix algebra
and polynomial arithmetic,
and should be easily implemented in most computational algebra systems.

\pagebreak
\begin{algorithm}
\label{algorithm:ell=3}
    A streamlined Dolgachev--Lehavi-style algorithm
    for \(\ell = 3\).
\begin{description}
\item[Input]    A genus 2 curve 
    \(\HC: Y^2 = F(X,Z) = \sum_{i=0}^6F_iX^iZ^{6-i} \)
    over~\(\field\)
    and a maximal Weil-isotropic subgroup \(S\) of \(\Jac{\HC}[3]\),
    its elements defined over \(\field\) 
    and presented as in~\S\ref{sec:Mumford}.
\item[Output]
    A genus 2 curve \(\XC/\field\)
    such that there exists an isogeny \(\phi: \Jac{\HC}\to\Jac{\XC}\)
    with kernel \(S\)
    (the curve \(\XC\) is computed up to a quadratic twist,
    so the isogeny may only be defined over a quadratic extension of
    \(\field\)).
\item[1]
    Compute a minimal subset \(S^{\pm}\) of \(S\)
    such that 
    \(S = \{e: e \in S^\pm\} \cup \{-e: e \in S^\pm\}\cup\{0\}\)
    (then \(\{\Secant{e}: e \in S^\pm\} = \{\Secant{e}: e \in
    S\setminus\{0\}\}\); this avoids redundancy in Steps 2 and 3).
\item[2]
    For each \(e\) in \(S^\pm\),
    compute the vector \(v_e = (\gamma_0(e),\ldots,\gamma_6(e))\)
    using the formul\ae{} in Proposition~\ref{proposition:H-Le-ell=3}.
\item[3]
    Compute vectors
    \(n_i = (\nu_{i,0},\ldots,\nu_{i,6})\) 
    such that \(\{n_0,n_1,n_2,(F_j:0\le j\le 6)\}\)
    is a basis for the (left) kernel of
    the \(7\times4\) matrix 
    \((v_e^t: e \in S^\pm) \).
    Set
    \[
        \Map_i = \sum_{j=0}^{6}\nu_{i,j}X^jZ^{6-j} 
        \quad 
        \text{for}\ 
        0 \le i \le 2
        .
    \] 
\item[4]
    For each \(0 \le i \le j \le 2\),
    compute the vector
    \(r_{i,j}\)
    of length \(6\)
    whose \(n^{\mathrm{th}}\)
    entry is the coefficient of \(x^{n-1}\)
    in \((\Phi_i\Phi_j)(x,1)\bmod F(x,1)\).
    If \(F_6 = 0\), then take the \(6^\mathrm{th}\)
    entry of \(r_{i,j}\) to be \(\nu_{i,0}\nu_{j,0}\):
    this allows us to correctly interpolate through the image
    of the Weierstrass point at infinity.
\item[5]
    Compute a generator \((q_{i,j}:0\le i \le j \le 2)\)
    for the (left) kernel of 
    the \(6\times 6\) matrix whose rows are the \(r_{i,j}\)
    for \(0 \le i \le j \le 2\).
    Set
    \[
        \qquad \qquad
        Q(V_0,V_1,V_2)
        :=
        q_{0,0}V_0^2
        + q_{0,1}V_0V_1
        + q_{0,2}V_0V_2
        + q_{1,1}V_1^2
        + q_{1,2}V_1V_2
        + q_{2,2}V_2^2
        .
    \]
\item[6]
    For each \(0 \le i \le j \le k \le 2\),
    compute the vector
    \(s_{i,j,k}\)
    of length \(6\)
    whose \(n^{\mathrm{th}}\)
    entry is the coefficient of \(x^{n-1}\)
    in \((\Phi_i\Phi_j\Phi_k)(x,1)\bmod F(x,1)\).
    If \(F_6 = 0\), then take the \(6^\mathrm{th}\)
    entry of \(s_{i,j,k}\) to be \(\nu_{i,0}\nu_{j,0}\nu_{k,0}\).
\item[7]    
    Compute any nontrivial element
    \((c_{i,j,k}:0\le i \le j \le k \le 2)\)
    of the (left) kernel of 
    the \(10\times 6\) matrix whose rows are the
    \(s_{i,j,k}\)
    for \(0 \le i \le j \le k \le 2\),
    and set
    \[
        C(V_0,V_1,V_2)
        :=
        \sum_{0\le i\le j\le k\le 2}c_{i,j,k}V_iV_jV_k
        .
    \]
\item[8]
    Return the result \(\XC\)
    of Algorithm~\ref{algorithm:curve-recovery}
    applied to \(\Conic = \variety{Q} \) 
    and \(\Cubic = \variety{C} \).
\end{description}
\end{algorithm}

\section{
    The algorithm in practice
}
\label{sec:example}

We conclude with an example for \(\ell = 3\).
To avoid a visually overwhelming mass of coefficients,
we will work over a small finite field;
the curve was chosen at random.

Consider the genus 2 curve over \(\FF_{997}\)
defined by
\[
    \HC: Y^2 = X^6 + 113X^5Z + 99X^4Z^2 + 363X^3Z^3 + 64X^2Z^4 +
    503XZ^5 + 630Z^6 .
\]    
Computing the zeta function of \(\HC\) (using Magma),
we see that its Weil polynomial is
\[
    P(T) = T^4 - 31T^3 + 54T^2 - 30907T + 994009 ,
\]    
so \(\Jac{\HC}\) is absolutely simple 
by the Howe--Zhu criterion~\cite[Theorem 6]{Howe--Zhu}.
The elements \(D_1 = \Mumford{x^2 + 392x + 208, 579x + 603, 2}\)
and
\(D_2 = \Mumford{x^2 + 48x + 527, 918x + 832, 2}\)
of \(\Jac{\HC}\)
have order \(3\), 
and
\(S = \subgroup{D_1,D_2}\)
is a maximal \(3\)-Weil isotropic subgroup of \(\Jac{\HC}[3]\).
Applying Algorithm~\ref{algorithm:ell=3},
we may take
\[
    S^\pm = \left\{\!\!\!\!
        \begin{array}{l}
            \Mumford{x^2 + 392x + 208, 579x + 603, 2}, 
            \Mumford{x^2 + 48x + 527, 918x + 832, 2}, 
            \\
            \Mumford{x^2 + 428x + 880, 252x + 901, 2}, 
            \Mumford{x^2 + 348x + 292, 596x + 269, 2}
        \end{array}
    \!\!\!\!\right\}
\]
in Step 1. Equation~\eqref{eq:ell=3-Hplane}
shows that the hyperplane \(\Hplane \subset \PP{6}\) is defined by
\[
    \Hplane: 
    630U_0 + 503U_1 + 64U_2 + 363U_3 + 99U_4 + 113U_5 + U_6
    = 0
    ,
\]
so the matrix in Step 3 is
\[
    \left(\begin{array}{rrrr}
    234 & 319 & 906 & 896 \\
    780 & 16 & 29 & 754 \\
    500 & 565 & 703 & 398 \\
    680 & 329 & 823 & 248 \\
    324 & 68 & 779 & 868 \\
    742 & 416 & 468 & 392 \\
    664 & 395 & 698 & 952
    \end{array}\right)
    ;
\]
computing kernel vectors, we take
\[
    \begin{array}{r@{\;=\;}l}
        \Phi_0 & 121X^6 + 742X^5Z + 549X^4Z^2 + XZ^5   ,
        \\
        \Phi_1 & 285X^6 + 642X^5Z + 332X^4Z^2 + X^2Z^4 ,
        \\
        \Phi_2 & 889X^6 + 701X^5Z + 454X^4Z^2 + X^3Z^3 .
    \end{array}
\]
The quadratic form of Step 5 is then
\[
    Q(V_0,V_1,V_2)
    =
    V_0^2 + 52V_0V_1 + 361V_1^2 + 548V_0V_2 + 715V_1V_2 + 296V_2^2
    ,
\]    
and we may take the cubic form in Step 7 to be
\[
    C(V_0,V_1,V_2) = 
    V_0^3 + 167V_1^3 + 149V_0V_1V_2 + 836V_1^2V_2 + 885V_0V_2^2 + 538V_1V_2^2 + 294V_2^3
    .
\]    
We now apply Algorithm~\ref{algorithm:curve-recovery}
to \(\Conic: Q(V_0,V_1,V_2) = 0\) 
and \(\Cubic: C(V_0,V_1,V_2) = 0\).
The conic~\(\Conic\) is nonsingular,
and
\(\HC\) has a rational Weierstrass point \((-76:0:1)\)
mapping to the point \((-36:-80:1)\) on \(\Conic\).
The associated parametrization \(\PP{1} \to \Conic\)
is defined by
\[
    (X:Z) 
    \longmapsto 
    ( 
        36X^2 + 781XZ + 109Z^2:
        80X^2 + 865XZ + 17Z^2:
        996X^2 + 945XZ + 636Z^2
    )
    ;
\]
substituting its defining polynomials into \(C\),
we find that \(\XC\) has a model
\[
    \XC: Y^2 = 118X^5Z + 183X^4Z^2 + 613X^3Z^3 + 35X^2Z^4 + 174XZ^5 + 474Z^6
    .
\]    
In fact, this is the quadratic twist of the true \(\XC\):
explicit calculation shows that its Weil polynomial is \(P(-T)\).


\begin{thebibliography}{XX}

\bibitem{avIsogenies}
    G.~Bisson, R.~Cosset, and D.~Robert,
    \emph{avIsogenies: a library for computing isogenies between abelian
    varieties}.  
    \ \url{avisogenies.gforge.inria.fr}

\bibitem{Magma-JSC}
    W.~Bosma, C.~Playoust, and J.~J.~Cannon,
    \emph{The Magma algebra system. I. The user language}. 
    J.~Symbolic~Comput.~\textbf{24} (1997), 235--265

\bibitem{Bost--Mestre}
    J.-B.~Bost and J.-F.~Mestre,
    \emph{Moyenne arithm\'etico-g\'eom\'etrique et p\'eriodes des courbes de
    genre 1 et 2}.
    Gaz.~Math.~\textbf{38} (1988), 36--64

\bibitem{Cardona--Quer}
    G.~Cardona and J.~Quer,
    \emph{Field of moduli and field of definition for curves of genus 2}.
    In \emph{Computational aspects of algebraic curves}, 
    Lecture Notes Ser.~Comput.~\textbf{13}, 71--83.
    World Sci. Publ., Hackensack, NJ, 2005.

\bibitem{Cassels--Flynn}
    J.~W.~S Cassels and E.~V.~Flynn,
    \emph{Prolegomena to a middlebrow arithmetic of curves of genus 2}.
    London Mathematical Society Lecture Note Series~\textbf{230},
    Cambridge University Press, Cambridge (1996)

\bibitem{Cremona--Rusin}
    J.~E.~Cremona and D.~Rusin,
    \emph{Efficient solution of rational conics}.
    Math.~Comp.~\textbf{72} 473 (2003), 1417--1441

\bibitem{Dolgachev--Lehavi}
	I.~Dolgachev and D.~Lehavi,
	\emph{On isogenous principally polarized Abelian surfaces}.
    In \emph{Curves and abelian varieties}, 
    Contemp.~Math.~\textbf{465} (2008), 51--69

\bibitem{Howe--Zhu}
    E.~W.~Howe and H.~J.~Zhu,
    \emph{On the existence of absolutely simple abelian varieties
    of a given dimension over an arbitrary field}.
    J.~Number Theory~\textbf{92} (2002), 139--163

\bibitem{Hudson}
    R.~W.~H.~T.~Hudson,
    \emph{Kummer's quartic surface}.
    Cambridge University Press, Cambridge (1990)

\bibitem{Lubicz--Robert}
    D.~Lubicz and D.~Robert,
    \emph{Computing isogenies between Abelian Varieties}.
    To appear in Compos.~Math.
    \url{hal.inria.fr/hal-00446062/en}

\bibitem{Magma}
    The Magma computational algebra system.
    \ \url{magma.maths.usyd.edu.au}

\bibitem{Mestre}
    J.-F.~Mestre,
    \emph{Construction de courbes de genre 2 \`a partir de leurs
    modules}. 
    In \emph{Effective methods in algebraic geometry (Castiglioncello,
    1990)}, Progr.~Math.~\textbf{94}, 313--334. 
    Birkh\"auser Boston, Boston, MA, 1991.

\bibitem{SAGE}
    W.~A.~Stein et al.,
    \emph{Sage Mathematics Software}.
    The Sage Development Team, 
    \url{www.sagemath.org}

\bibitem{Weil}
    A.~Weil,
    \emph{Zum Beweis des Torellischen Satzes}.
    Nachr.~Akad.~Wiss.~G\"ottingen.~Math.-Phys.~Kl.~IIa.~\textbf{1957} 
    (1957), 33--53

\end{thebibliography}
\end{document}